\documentclass[letterpaper]{article}

\author{Benjamin Antieau\footnote{\emph{E-mail:} \texttt{benjamin.antieau@gmail.com}.}}

\usepackage[letterpaper,left=1.5in,right=1.5in,top=1.5in,bottom=1.5in]{geometry}

\newcommand{\myauthor}{Benjamin Antieau}
\newcommand{\mytitle}{A local-global principle for the telescope conjecture}
\newcommand{\pdftitle}{\mytitle}

\title{\mytitle}

\usepackage[pdfstartview=FitH,
            pdfauthor={\myauthor},
            pdftitle={\pdftitle},
            colorlinks,
            linkcolor=reference,
            citecolor=citation,
            urlcolor=e-mail,
            ]{hyperref}

\usepackage{mathrsfs}
\usepackage[mathscr]{euscript}
\usepackage{amsmath}
\usepackage{amscd}
\usepackage{amsbsy}
\usepackage{amssymb}
\usepackage{microtype}
\usepackage{enumerate}
\usepackage[all,cmtip]{xy}
\usepackage{tikz}
\usetikzlibrary{matrix,arrows}
\usepackage{dsfont}
\usepackage[abbrev,lite]{amsrefs}
\usepackage{amsthm}

\usepackage{color}
\definecolor{todo}{rgb}{1,0,0}
\definecolor{conditional}{rgb}{0,1,0}
\definecolor{e-mail}{rgb}{0,.40,.80}
\definecolor{reference}{rgb}{.20,.60,.22}
\definecolor{mrnumber}{rgb}{.80,.40,0}
\definecolor{citation}{rgb}{0,.40,.80}

\DeclareMathOperator{\Ho}{Ho}

\newcommand{\perf}{\mathrm{perf}}

\DeclareMathOperator{\id}{id}

\DeclareMathOperator*{\colim}{colim}
\newcommand{\desc}{\mathrm{desc}}

\newcommand{\rwe}{\tilde{\rightarrow}}

\newcommand{\we}{\simeq}
\newcommand{\iso}{\cong}

\newcommand{\stovicek}{\v{S}\v{t}ov{\'\i}\v{c}ek}

\newcommand{\st}{\mathrm{st}}

\newcommand{\qc}{\mathrm{qc}}

\newcommand{\ccc}{\mathrm{c}}

\newcommand{\Prl}{\mathrm{Pr}^{\mathrm{L}}}

\newcommand{\Mod}{\mathrm{Mod}}
\newcommand{\Perf}{\mathrm{Perf}}

\newcommand{\Cat}{\mathrm{Cat}}

\DeclareMathOperator{\Spec}{Spec}

\DeclareMathOperator{\Hoh}{H}

\newcommand{\Map}{\mathrm{Map}} 

\newcommand{\StCat}{\mathscr{C}\mathrm{at}}

\newcommand{\StMod}{\mathscr{M}\mathrm{od}}

\newcommand{\et}{\mathrm{\acute{e}t}}

\DeclareMathOperator{\Br}{Br}

\newcommand{\Brm}{\mathrm{B}}

\newcommand{\Drm}{\mathrm{D}}

\newcommand{\Kscr}{\mathscr{K}}

\newcommand{\Oscr}{\mathscr{O}}
\newcommand{\Tscr}{\mathscr{T}}

\newcommand{\Mscr}{\mathscr{M}}

\newcommand{\Ascr}{\mathscr{A}}

\newcommand{\Cscr}{\mathscr{C}} 
\newcommand{\Dscr}{\mathscr{D}}

\newcommand{\Xscr}{\mathscr{X}}
\newcommand{\Gscr}{\mathscr{G}}
\newcommand{\Uscr}{\mathscr{U}}

\newcommand{\CC}{\mathds{C}}

\newcommand{\ZZ}{\mathds{Z}}

\newcommand{\PP}{\mathds{P}}

\renewcommand{\SS}{\mathds{S}}

\theoremstyle{plain}
\newtheorem{theorem}{Theorem}[section]
\newtheorem{lemma}[theorem]{Lemma}
\newtheorem{proposition}[theorem]{Proposition}

\newtheorem{conjecture}[theorem]{Conjecture}
\newtheorem{corollary}[theorem]{Corollary}

\theoremstyle{definition}
\newtheorem{definition}[theorem]{Definition}

\newtheorem{example}[theorem]{Example}
\newtheorem{question}[theorem]{Question}

\newtheorem{remark}[theorem]{Remark}

\let\oldmarginpar\marginpar
\renewcommand\marginpar[1]{\-\oldmarginpar[\raggedleft\footnotesize #1]%
{\raggedright\footnotesize #1}}

\usepackage{txfonts}

\begin{document}
\maketitle

\begin{abstract}

    \noindent
    We prove an \'etale local-global principle for the telescope conjecture and use it to
    show that the telescope conjecture holds for derived categories of Azumaya algebras on
    noetherian schemes as well as for many classifying stacks and gerbes. This specializes
    to give another proof of the fact that the telescope conjecture holds for noetherian
    schemes.

    \paragraph{Key Words}
    Telescope conjecture, derived categories, Azumaya algebras.

    \paragraph{Mathematics Subject Classification 2010}
    Primary:
    \href{http://www.ams.org/mathscinet/msc/msc2010.html?t=16Exx&btn=Current}{16E35},
    \href{http://www.ams.org/mathscinet/msc/msc2010.html?t=18Exx&btn=Current}{18E30}.
    Secondary:
    \href{http://www.ams.org/mathscinet/msc/msc2010.html?t=14Fxx&btn=Current}{14F22},
    \href{http://www.ams.org/mathscinet/msc/msc2010.html?t=18Gxx&btn=Current}{18G55}.

\end{abstract}

\section{Introduction}

Let $\Tscr$ be a compactly generated triangulated category with all coproducts. Recall that a (Bousfield)
localization of $\Tscr$ consists of a second triangulated category $\Tscr'$ and a pair of
adjoint functors
\begin{equation*}
    j:\Tscr\rightleftarrows\Tscr':j_\rho
\end{equation*}
such that $j_\rho$ is fully faithful, $j$ being the left adjoint and $j_\rho$ the right. The
associated localization functor is the composition $j_\rho\circ j$. A
localization is called smashing if $j_\rho$ preserves
coproducts, which is equivalent to saying that $j_\rho\circ j$ preserves coproducts.

\begin{conjecture}[Triangulated telescope conjecture]
    If $j:\Tscr\rightleftarrows\Tscr':j_\rho$ is a smashing localization, then $\ker(j)$ is generated
    by objects that are compact in $\Tscr$, where $\ker(j)$ is the full subcategory of
    $\Tscr$ consisting of objects $x$ such that $j(x)\we 0$.
\end{conjecture}

As a simple but crucial example, let $Z$ be a closed subscheme of $X$ (which we assume to be
quasi-compact and quasi-separated) defined by $n$
equations $f_1,\ldots,f_n$. Write $U$ for the complement of $Z$ in $X$. Then,
the restriction functor $\Drm_{\qc}(X)\rightarrow\Drm_{\qc}(U)$ is a smashing
localization. The conjecture can be verified directly in this case as follows. Let $K_i$ be
the perfect complex $\Oscr_X\xrightarrow{f_i}\Oscr_X$, and let $K=K_1\otimes\cdots\otimes
K_n$ (the Koszul complex). Then, $K$ is a compact generator of the kernel of the
localization functor. This was first observed by B\"okstedt and
Neeman~\cite{bokstedt-neeman} in the affine case. See~\cite{ag}*{Proposition 6.9} for the
general case.

The telescope conjecture is not really a conjecture, as it is known to be
false in certain circumstances, even for the derived categories $\Drm(R)$ of commutative
rings $R$. The first example was given by Keller~\cite{keller-smashing}, and more recent
examples, of certain
dimension $2$ valuation rings, were given by Krause and
\stovicek~\cite{krause-stovicek}*{Example 7.8}. Nevertheless, there is a great deal of
interest in cases when it does hold, because it relates the classification of smashing
localizations of $\Tscr$ to the classification of thick subcategories of
$\Tscr^\ccc$, the full subcategory of compact objects, and the latter classification problem
is sometimes tractable.

To summarize what is known at present,
Hopkins and Neeman~\cite{neeman-chromatic} gave the first results, establishing the conjecture for $\Drm(R)$, the derived
category of a noetherian commutative ring. As a consequence, one finds that there is a
bijection between the smashing localizations of $\Drm(R)$ (up to equivalence), thick
subcategories of the triangulated category $\Drm_{\perf}(R)$ of perfect complexes on $R$, and
specialization-closed subsets of $\Spec(R)$.
In the non-noetherian case, Dwyer and
Palmieri~\cite{dwyer-palmieri} showed that the conjecture holds for the derived categories of truncated
polynomial algebras in countably many generators, while Stevenson~\cite{stevenson-flat}
established the conjecture for absolutely flat rings.

In the noncommutative case, 
Br\"uning~\cite{bruning} proved the conjecture for the derived categories of finite
dimensional hereditary algebras of finite representation type over a field, a result which was then
extended to all finite dimensional hereditary algebras over a field by
Krause-\stovicek~\cite{krause-stovicek}. In particular, the telescope conjecture holds for
$\Drm_{\qc}(\PP^1_k)$, making $\PP^1_k$ the only non-affine variety for which this form of the
telescope conjecture is known to hold.

In another direction, Stevenson proved the conjecture for the singularity categories of
noetherian rings with hypersurface singularities
in~\cite{stevenson-singularity} and for quotients of regular local rings by regular
sequences.

The triangulated telescope conjecture was originally formulated for the stable homotopy
category $\mathrm{SH}$ by Bousfield~\cite{bousfield-localization}*{Conjecture 3.4}. In the
form written here it was given by Ravenel~\cite{ravenel-localization}*{Conjecture 1.33}. The
reason for its importance in stable homotopy theory is that if true for the $p$-local stable
homotopy category $\mathrm{SH}_{(p)}$, it would give a concrete way of computing
the $K(n)$-localization of a space via a telescope construction, which is a certain homotopy
colimit. Specifically,
the thick subcategories of the triangulated category of $p$-local finite spectra
$\mathrm{SH}^{\mathrm{fin}}_{(p)}$ are known: they are precisely the thick subcategories
given by the kernels of $E(n)$-localization for some $n$. Writing $L_n^t$ for telescopic
localization, which for a $p$-local finite spectrum of type at least $n$ can be described as
$\nu_n$-localization, the kernel of any localization $L$ satisfies
\begin{equation*}
    \ker(L_n^t)\subseteq\ker(L)\subseteq\ker(L_{E(n)})
\end{equation*}
for some uniquely determined non-negative integer $n$.
The telescope conjecture would say that these are equalities.
The current state of the telescope conjecture for the stable homotopy category seems
unclear. Apparently, it is widely believed to be false, and potential counterexamples have
even been produced at various points, but a proof that it is false remains elusive.

There is another version of the telescope conjecture suitable for when a
$\otimes$-triangulated category $\Tscr$ acts on a triangulated category $\Uscr$.
Again, we require $\Tscr$ and $\Uscr$ to have all coproducts and to be compactly generated.
We also require the tensor product map
\begin{equation*}
    \otimes:\Tscr\times\Uscr\rightarrow\Uscr
\end{equation*}
to preserve coproducts in each variable. A localizing subcategory of $\Uscr$ will be called
$\Tscr$-closed if it is closed under tensor products with $\Tscr$.

\begin{conjecture}[Tensor telescope conjecture]
    If $j:\Uscr\rightleftarrows\Uscr':j_\rho$ is a smashing localization where $\ker(j)$ is
    $\Tscr$-closed, then $\ker(j)$ is generated by objects that are compact in $\Uscr$.
\end{conjecture}

We will say that the $\otimes$-telescope conjecture holds for $\Uscr$ under the action of
$\Tscr$ when the conjecture is verified. When $\Tscr=\Uscr$, we will simply say that the $\otimes$-telescope conjecture
holds for $\Tscr$. When the unit of $\Tscr$ is a compact generator,
as is the case for the stable homotopy category $\mathrm{SH}$ or the derived
category $\Drm(R)$ of a commutative ring $R$, the $\otimes$-telescope conjecture for $\Uscr$
under the action of $\Tscr$ is equivalent to the triangulated telescope conjecture, since in
that case \emph{every} localizing subcategory of $\Uscr$ is closed under tensoring with
objects of $\Tscr$.

The tensor telescope conjecture was stated in this form by
Stevenson~\cite{stevenson-support}, generalizing the situation where $\Uscr=\Tscr$ considered
previously. An example of why it is
useful to consider the more general situation is that if $\alpha\in\Br(X)$, then
$\Drm_{\qc}(X,\alpha)$ is not a
$\otimes$-category. But, nevertheless, as one result of our paper, if $X$ is noetherian,
then the $\otimes$-telescope conjecture holds for $\Drm_{\qc}(X,\alpha)$ under the action of
$\Drm_{\qc}(X)$.

In~\cite{atjlss}, the authors show that the $\otimes$-telescope conjecture holds for the derived categories
of noetherian formal schemes, extending the Hopkins-Neeman result in particular to the
derived categories of quasi-coherent sheaves on noetherian schemes.
Balmer and Favi~\cite{balmer-favi}
established a local-global principle under which the $\otimes$-telescope holds globally if it holds
Zariski locally on Balmer's spectrum for tensor triangulated categories~\cite{balmer}. Their
work gives another proof of the $\otimes$-telescope conjecture for the derived categories of
quasi-coherent sheaves $\Drm_{\qc}(X)$ on noetherian schemes.
Hovey, Palmieri, and Strickland~\cite{hovey-palmieri-strickland} gave a new proof of Neeman's result, using the equivalence
between the triangulated tensor conjecture and the $\otimes$-telescope conjecture for $\Drm(R)$
when $R$ is noetherian. Their methods also prove the $\otimes$-telescope conjecture for
comodules over a finite-dimensional Hopf algebra.
Benson, Iyengar, and Krause established the $\otimes$-telescope conjecture for the
homotopy category of injective complexes and for the stable category of a
finite group in~\cite{benson-iyengar-krause}.

For a quasi-compact and quasi-separated scheme,
Thomason~\cite{thomason-triangulated} classified the thick $\otimes$-subcategories of $\Drm_{\perf}(X)$, the triangulated
category of perfect complexes on $X$. By the results of~\cite{atjlss}
and~\cite{balmer-favi}, there is a nice description of all smashing $\otimes$-localizations of
$\Drm(X)$ for $X$ noetherian. In particular, to any such localization there is a uniquely
defined specialization-closed subset of $X$.
This subset is precisely the locus where the objects of $\ker(j)$ are supported.

Stevenson~\cite{stevenson-support} considered the theory of supports that arises when $\Tscr$ acts on $\Uscr$ and
used this to give yet another
proof of the tensor telescope conjecture for the derived categories of noetherian schemes.
Stevenson's proof is conceptually satisfying as it proceeds by actually classifying the
tensor closed subcategories, yielding a proof closer in spirit to Neeman's proof in the
affine case. Dell'Ambrogio and Stevenson~\cite{dellambrogio-stevenson} proved the $\otimes$-telescope
conjecture for quasi-projective varieties and for weighted projective spaces by considering
a graded version of the telescope hypothesis and then using support theory.

\vspace{10pt}
The work of Balmer and Favi shows in a great deal of generality that the $\otimes$-telescope conjecture holds for $\Tscr$ acting on itself when it holds
locally on the Balmer spectrum of $\Tscr^\ccc$. 
The purpose of this paper is to establish a new local-global principle for the
telescope conjecture. Our principle differs in three important ways from theirs.
First, it holds for \'etale covers not just Zariski covers. Second,
it works for the action of $\Tscr$ on $\Uscr$, allowing one to establish
telescopy in noncommutative situations such as for Azumaya algebras. This perspective is
present in Stevenson~\cite{stevenson-support} as well. Third, it requires as
input not triangulated categories but enhancements such as
stable $\infty$-categories. This restriction is not a barrier for any
foreseeable application.

Our methods use in a crucial way the notion of a stack of stable presentable
$\infty$-categories over a scheme $X$, as studied in~\cite{dag11} and~\cite{ag}. The dg
category approach to these ideas can be found in~\cite{toen-derived}. We very briefly describe
this theory here, referring the reader more generally to~\cite{ag}*{Section 6} and the
references there. These stacks provide one method of giving sense to the nonsense notion of
a stack of triangulated categories.

We fix a base connective commutative ring spectrum $R$. This might be the Eilenberg-MacLane
spectrum of an ordinary commutative ring or of a simplicial commutative ring. For our
applications, we use only ordinary commutative rings, but it seems relevant to note that the
theorems hold for quasi-compact and quasi-separated derived schemes, which are schemes with
sheaves of local connective commutative ring spectra.

A stable $\infty$-category $\Cscr$
is an $\infty$-category that has a $0$ object, that has fiber and cofiber sequences, and in
which fiber and cofiber sequences agree. The homotopy category $\Ho(\Cscr)$ of $\Cscr$
is naturally a triangulated category. By~\cite{ha}*{Corollary 1.4.4.2}, a stable $\infty$-category $\Cscr$ is presentable if
$\Ho(\Cscr)$ has all coproducts, has hom \emph{sets}, and has a $\kappa$-compact generator for some regular
cardinal $\kappa$.

When $A$ is an $A_\infty$-algebra spectrum (such as the Eilenberg-MacLane spectrum of an
ordinary associative algebra), $\Mod_A$ is a stable presentable $\infty$-category, with
homotopy category $\Drm(A)$. For example if $\SS$ is the sphere spectrum, then $\Mod_\SS$ is
an $\infty$-categorical enhancement of the triangulated stable homotopy category. For a
quasi-compact and quasi-separated scheme $X$, there is a stable presentable
$\infty$-category that we will denote by $\Mod_X$ with $\Ho(\Mod_X)=\Drm_{\qc}(X)$, the
triangulated category of complexes of $\Oscr_X$-modules with quasi-coherent cohomology
sheaves.

If $S$ is a connective commutative $R$-algebra, then an $S$-linear category is a stable
presentable $\infty$-category enriched over $\Mod_S$, the stable presentable
$\infty$-category of $S$-module spectra. These objects can be realized as the (left) modules
for the commutative ring object $\Mod_S$ in the $\infty$-category $\Prl_{\st}$. We denote
this category by $\Cat_S=\Mod_{\Mod_S}(\Prl_{\st})$.

An $S$-linear category with \'etale hyperdescent is an $S$-linear category $\Mod_S^\alpha$ such that for any
connective commutative $S$-algebra $T$ and any \'etale hypercover $\Spec
T^\bullet\rightarrow\Spec T$, the induced map
\begin{equation*}
    \Mod_T\otimes_{\Mod_S}\Mod_S^\alpha\rightarrow\lim_\Delta\Mod_{T^\bullet}\otimes_{\Mod_S}\Mod_S^\alpha
\end{equation*}
is an equivalence. These define a full subcategory $\Cat_S^\desc$ of $\Cat_S$. It is
an important fact that these glue together to form a stack $\StCat^\desc$
(see~\cite{dag11}*{Theorem 7.5}).

Let $X$ be an $R$-scheme (which might be derived).
An \'etale  hyperstack (henceforth just a stack) of linear categories on $X$ is a map of
stacks $\StMod^\alpha:X\rightarrow\StCat^\desc$ over $\Spec R$. Loosely speaking, $\StMod^\alpha$
assigns to any affine $\Spec S\rightarrow X$ a stable presentable $S$-linear category
$\Mod^\alpha_S$ and to any map $f:\Spec T\rightarrow\Spec S$ a pull-back map
$f^*:\Mod^\alpha_S\rightarrow\Mod^\alpha_T$ in such a way that if $\Spec T^\bullet\rightarrow\Spec S$ is an \'etale
hypercover, then the associated map
\begin{equation*}
    \Mod^\alpha_S\rightarrow\lim_\Delta\Mod^\alpha_{T^\bullet}
\end{equation*}
is an equivalence of stable presentable $S$-linear categories. The affines here are $\Spec
S$ for all connective commutative $R$-algebras, a class that includes the Eilenberg-MacLane spectra
of all ordinary $\pi_0R$-algebras. For us, $R$ itself will be such an Eilenberg-MacLane
spectra and no truly derived schemes will arise in the paper.

The $\infty$-category of global sections of
$\StMod^\alpha$ is
\begin{equation*}
    \Mod^\alpha_X=\lim_{\Spec S\rightarrow X}\Mod^\alpha_S,
\end{equation*}
where the limit is computed in $\Cat_R$.

\begin{example}
    The stack that assigns to each $\Spec S\rightarrow X$ the stable presentable
    $\infty$-category $\Mod_S$ of $S$-module spectra will be written $\StMod^{\Oscr}$. This can be thought of as
    the stack of complexes of $\Oscr_X$-modules with quasi-coherent cohomology. The homotopy category of
    $\Mod_X=\Mod_X^\Oscr$ is $\Drm_{\qc}(X)$.
\end{example}

The stable $\infty$-category $\Mod_X$ is symmetric monoidal, and any other category of
global sections $\Mod_X^\alpha$ comes with a natural action of $\Mod_X$. We say that
$\Mod^\alpha_X$ satisfies the $\Mod_X$-linear telescope hypothesis if the kernel of any $\Mod_X$-linear
smashing localization is generated by compact objects of $\Mod^\alpha_X$. When
$\alpha=\Oscr$, this is the $\infty$-categorical analogue of the $\otimes$-triangulated
telescope conjecture.

The local-global principle of the title is encoded in the following result.

\begin{theorem}
    Let $X$ be a quasi-compact and quasi-separated scheme, and suppose that
    $\StMod^\alpha$ is a stack of linear categories on $X$. If
    there is an \'etale cover $U\rightarrow X$ such that $\Mod^\alpha_U$ satisfies the
    $\Mod_U$-linear
    telescope hypothesis, then $\Mod^\alpha_X$ satisfies the $\Mod_X$-linear telescope hypothesis.
\end{theorem}

Of course, one might wonder how this statement translates into the original language of
triangulated categories. This is spelled out in detail in Section~\ref{sec:triangulated}:
the triangulated versions are equivalent to the $\infty$-categorical versions.
As a consequence of the theorem, we prove the telescope hypothesis in the following
situations\footnote{These are stated in the body of the paper in their $\infty$-categorical
forms. We translate them here into the world of triangulated categories.}.
\begin{enumerate}
    \item   The $\otimes$-telescope conjecture holds for $\Drm_{\qc}(X,\alpha)$ under the
        action of $\Drm_{\qc}(X)$, where $\Drm_{\qc}(X,\alpha)$ is the $\alpha$-twisted derived category of a noetherian
        scheme, for $\alpha\in\Br(X)$. Proving this result was the original motivation for
        the project. Even for $X$ affine this was unknown. When $X$ is
        affine and $\alpha=0$, this was Neeman's result. For $X$ a general noetherian scheme
        and $\alpha=0$, it has been proven by~\cite{atjlss}, \cite{balmer-favi},
        \cite{hovey-palmieri-strickland}, and \cite{stevenson-support}. Thus, we find a
        fifth proof, most similar in spirit to that of Balmer and Favi.
    \item   The $\otimes$-telescope conjecture holds for $\Drm_{\qc}(\Brm\Gscr)$ under the
        action of $\Drm_{\qc}(X)$, when $\Brm\Gscr$ is the classifying stack of
        a finite tame \'etale group scheme $\Gscr$ over a noetherian
        scheme $X$. Since $\Drm_{\qc}(\Brm\Gscr)$ is itself a $\otimes$-triangulated category and
        there is a $\otimes$-triangulated pullback functor
        $\Drm_{\qc}(X)\rightarrow\Drm_{\qc}(\Brm\Gscr)$, the $\otimes$-telescope conjecture holds for
        $\Drm_{\qc}(\Brm\Gscr)$ acting on itself as well. A similar comment applies in each of the
        next cases.
    \item   The $\otimes$-telescope conjecture holds for $\Drm_{\qc}(\Brm\Ascr)$ under the action of
        $\Drm_{\qc}(X)$, when
        $\Ascr$ is a finite abelian group scheme over a noetherian scheme $X$.
    \item   The $\otimes$-telescope conjecture holds for $\Drm_{\qc}(\Xscr)$ under the action of
        $\Drm_{\qc}(X)$, where
        $\Xscr\rightarrow X$ is a finite abelian gerbe over a noetherian scheme.
\end{enumerate}

Besides these results, we give several examples throughout of new cases of the telescope
conjecture. For example, in Example~\ref{ex:dg}, we give what we believe to be the first
example where telescopy holds for the derived category of a dg algebra that is not derived
Morita equivalent to an ordinary algebra. We also show that when $X$ is quasi-compact and
quasi-separated, the $\otimes$-closed smashing localizing subcategories of
$\Drm_{\qc}(X,\alpha)$ under the action of $\Drm_{\qc}(X)$ correspond bijectively to the
specialization closed subsets of $X$ that can be written as the union of closed subsets with
quasi-compact complement.

We end the paper by establishing the following classification theorem using recent work of
Dubey and Mallick~\cite{dubey-mallick}.

\begin{theorem}
    Let $X$ be a smooth scheme of finite type over a field $k$, and let $\Gscr\rightarrow X$
    be a finite \'etale group scheme of order prime to the characteristic of $k$. Suppose that $\Xscr\rightarrow X$ is a
    $\Gscr$-gerbe (a stack over $X$ \'etale locally equivalent to $\Brm\Gscr$). Then, there is a
    bijection between the set of $\otimes$-closed smashing localizations of
    $\Drm_{\qc}(\Xscr)$ and the specialization closed subsets of $X$.
\end{theorem}

In Section~\ref{sec:compact} we prove a local-global principle for
the property of being compactly generated. Section~\ref{sec:telescopy} contains the main definitions,
of the telescope hypothesis, the linear telescope hypothesis, and the stacky telescope
hypothesis. At the end, we prove a key theorem that says that the stacky telescope
hypothesis is equivalent to the linear telescope hypothesis. In
Section~\ref{sec:triangulated} the
$\infty$-categorical telescope hypotheses are compared to the triangulated versions, and are
shown to be equivalent where appropriate. Section~\ref{sec:localglobal} contains the main theorem, the
local-global principle, as well as the consequences for schemes and Azumaya algebras.
Finally, in Section~\ref{sec:gerbes}, we prove the linear telescope conjecture for
classifying stacks of finite \'etale group schemes in the tame case, for finite abelian
group schemes, and for gerbes over these.

\subsection{Acknowledgments}

I thank Paul Balmer and David Gepner for conversations about telescopy and Greg Stevenson
for an illuminating comment about when localizing subcategories can be lifted to a
model.

\section{The local-global principle for compact generation}\label{sec:compact}

We need the following result, which is not quite proved in the union of the
papers of Lurie~\cite{dag11}, To\"en~\cite{toen-derived}, and Antieau-Gepner~\cite{ag}.
The idea is due to B\"okstedt and Neeman~\cite{bokstedt-neeman}.

\begin{theorem}\label{thm:lgcptg}
    Let $\alpha:X\rightarrow\StCat^{\desc}$ classify a stack of linear categories $\StMod^\alpha$
    over $X$, where $X$ is a quasi-compact and quasi-separated scheme. If $\StMod^\alpha$ is
    \'etale locally compactly generated, then $\Mod^\alpha_X$ is compactly generated.
\end{theorem}

A special case of the telescope hypothesis is needed in the proof of the theorem.

\begin{lemma}
    Let $Z=\Spec S$ be an affine scheme and $W\subseteq Z$ a quasi-compact Zariski open inclusion.
    Let $\alpha:Z\rightarrow\StCat^{\desc}$ classify a stack of linear categories
    $\StMod^\alpha$. If $\Mod_Z^\alpha$ is compactly generated, then the
    kernel $\Mod_{Z,Z-W}^\alpha$
    of $\Mod_Z^\alpha\rightarrow\Mod_W^\alpha$ is compactly generated by compact objects of
    $\Mod_Z^\alpha$.
    \begin{proof}
        Because tensor products of stable presentable $\infty$-categories are computed as
        functors~\cite{ha}*{Proposition~6.3.1.16}, it follows that the exact sequence
        \begin{equation*}
            \Mod_{Z,Z-W}^\alpha\rightarrow\Mod_Z^\alpha\rightarrow\Mod_W^\alpha
        \end{equation*}
        is obtained from
        \begin{equation*}
            \Mod_{Z,Z-W}\rightarrow\Mod_Z\rightarrow\Mod_W
        \end{equation*}
        by tensoring with $\Mod_Z^\alpha$ over $\Mod_Z$. By~\cite{ag}*{Proposition~6.9},
        $\Mod_{Z,Z-W}$ is generated by a single compact object. Since, by hypothesis,
        $\Mod_Z^\alpha$ is compactly generated, it follows that
        \begin{equation*}
            \Mod_{Z,Z-W}^\alpha\we\Mod_{Z,Z-W}\otimes_{\Mod_Z}\Mod_Z^\alpha
        \end{equation*}
        is compactly generated (see~\cite{bgt1}*{Section 3.1}).
    \end{proof}
\end{lemma}

Say that an object of $\Mod_X^\alpha$ is perfect if for every $\Spec S\rightarrow X$ the
pullback $x_S$ is compact in $\Mod_S^\alpha$.

\begin{proposition}
    In the situation of the theorem, perfect objects of $\Mod_X^\alpha$ are compact.
    \begin{proof}
        First, this is true on
        affine schemes by definition. Second, if $X=U\cup V$ where $U$ and $V$ are open
        subschemes, and if it is true for $U$ and $V$ and $U\cap V$, then it is true for
        $X$. Indeed, in this case, $X$ is the \emph{finite} colimit $U\cap
        V\rightrightarrows U\coprod V\rightarrow X$. Thus, $\Mod_X^\alpha$ is the fiber in
        \begin{equation*}
            \Mod_X^\alpha\rightarrow\Mod_U^\alpha\times\Mod_V^\alpha\rightrightarrows\Mod_{U\cap
            V}^\alpha.
        \end{equation*}
        Given objects $x,y\in\Mod_X^\alpha$, this means that we can compute the mapping spectrum
        $\Map_X(x,y)$ as a limit
        \begin{equation*}
            \Map_X(x,y)\rightarrow\Map_U(x_U,y_U)\times\Map_V(x_V,y_V)\rightrightarrows\Map_{U\cap V}(x_{U\cap V},y_{U\cap V}).
        \end{equation*}
        If $x$ is perfect in $\Mod_X^\alpha$, then it is compact on $U$, $V$, and $U\cap V$ by
        hypothesis. Since filtered colimits commute with finite limits, it then follows that
        $x$ is compact in $X$, as desired.
        Finally, the result holds for arbitrary quasi-compact and quasi-separated schemes by
        the so-called reduction principle~\cite{bondal-vandenbergh}*{Proposition 3.3.1}.
    \end{proof}
\end{proposition}

\begin{proof}[Proof of Theorem~\ref{thm:lgcptg}]
    The proof is essentially a transcription of the proof of~\cite{ag}*{Theorem 6.11}, with
    a couple of alterations. The base case of the induction step is that if $X$ is affine then
    \'etale local compact generation implies global compact generation. Moreover, in that
    case any compact objects is perfect. This step is
    provided by~\cite{dag11}*{Theorem 6.1}. The compact generation of the kernels is
    provided by the lemma. The rest of the proof goes through, except that one
    lifts \emph{sets} of compact generators up to $X$ using the gluing methods
    of~\cite{ag}*{Theorem 6.11}. Details are left to the reader. The last step is to note
    that one has built up perfect objects, which are compact by the proposition.
\end{proof}

\begin{corollary}\label{cor:perfect}
    In the situation of the theorem, the compact objects of $\Mod_X^\alpha$ are precisely
    the perfect objects.
    \begin{proof}
        Since perfect objects are compact and generate $\Mod_X^\alpha$, the theorem of Ravenel and Neeman~\cite{neeman-1992}*{Theorem
        2.1} shows that the subcategory of compact objects of $\Mod_X^\alpha$ is the
        idempotent completion of the subcategory of perfect objects. But, the subcategory of
        perfect objects is already idempotent-complete, as can be seen by looking locally.
    \end{proof}
\end{corollary}

\section{Telescopy}\label{sec:telescopy}

A localization of a stable presentable $\infty$-category $\Cscr$ is an adjunction
\begin{equation*}
    j:\Cscr\rightleftarrows\Dscr:j_\rho
\end{equation*}
where the right adjoint $j_\rho$ is fully faithful.

Recall that if $\Mscr$ is a symmetric
monoidal stable presentable $\infty$-category, then we can consider ``modules'' for $\Mscr$,
which are stable presentable $\infty$-categories $\Cscr$ with a tensor product
$\otimes:\Mscr\times\Cscr\rightarrow\Cscr$ satisfying various nice properties, most
importantly the preservation of homotopy colimits in each variable. These
$\infty$-categories together with the $\otimes$-structure will be called $\Mscr$-linear
categories. By working in $\Prl$,
the symmetric monoidal $\infty$-category of presentable $\infty$-categories and right
adjoint functors, an $\Mscr$-linear category $\Cscr$ is precisely a (left) module for the
commutative algebra object $\Mscr$. See~\cite{ha}*{Section 6.3}.

If $\Cscr$ is an $\Mscr$-linear category, then a localization
$j:\Cscr\rightleftarrows\Dscr:j_\rho$ is $\Mscr$-linear if it is a localization in the
$\infty$-category of $\Mscr$-modules in $\Prl$. This can be checked in a more down-to-earth
way by showing that $\ker(j)$ is closed under tensor product with $\Mscr$.

A localization is
smashing if $j_\rho$ preserves small coproducts. Note that because these stable
$\infty$-categories are presentable, preserving coproducts is equivalent to preserving all
small colimits in the $\infty$-categorical sense, by~\cite{ha}*{Proposition 1.4.4.1}.

A localization of stacks consists of an adjunction
\begin{equation*}
    j:\StMod^\alpha\rightleftarrows\StMod^\beta:j_\rho,
\end{equation*}
where $\StMod^\alpha$ and $\StMod^\beta$ are stacks of linear categories, such that $j_\rho$ is fully
faithful. By definition, to give such an adjunction is to give a compatible family of
$S$-linear adjunctions
\begin{equation*}
    j_S:\Mod^\alpha_S\rightleftarrows\Mod^\beta_S:j_{S,\rho}
\end{equation*}
for every $\Spec S\rightarrow X$. Then, the functor $j_\rho$ is fully faithful if
each $j_{S,\rho}$ is fully faithful. The localization of stacks is smashing if each $j_{S,\rho}$ preserves
small coproducts.

Given a localization $j:\Cscr\rightarrow\Dscr$, there is a kernel $\ker(j)$, the full
subcategory of $\Cscr$ of objects $x$ such that $j(x)\we 0$. Given an $\Mscr$-linear
localization, the kernel $\ker(j)$ is itself $\Mscr$-linear. For a localization of stacks
$j:\StMod^\alpha\rightarrow\StMod^\beta$, the family of kernels determines itself a stack of
linear categories $\StMod^\gamma$ by setting $\Mod^\gamma_S=\ker(j_S)$. To see this, it
suffices to check when $X=\Spec S$, in other words in the case of $S$-linear categories with
descent. But, the kernel is a limit in $\Cat_S^{\desc}$, so it can be computed \'etale
locally, since limits commute.

\begin{definition}

    Let $\Cscr$ be a compactly generated stable presentable $\infty$-category. Then, $\Cscr$
    satisfies the telescope hypothesis (TH) if the kernel of every smashing localization $j:\Cscr\rightarrow\Dscr$
    is generated by compact objects of $\Cscr$.

    Now, suppose that $\Mscr$ is a compactly generated symmetric monoidal stable presentable
    $\infty$-category and that $\Cscr$ is a compactly generated $\Mscr$-linear category.
    Say that $\Cscr$ satisfies the $\Mscr$-linear telescope
    hypothesis (LTH) if the kernel of every $\Mscr$-linear smashing localization
    $j:\Cscr\rightarrow\Dscr$ is generated by compact objects of $\Cscr$.

    Finally, suppose that $X$ is an \'etale sheaf over $R$, and let
    $\alpha:X\rightarrow\StCat^\desc$ classify a stack $\StMod^\alpha$ of linear
    categories. Say that $\StMod^\alpha$ satisfies the stacky telescope hypothesis (STH) if for every
    smashing localization of stacks $\StMod^\alpha\rightarrow\StMod^\beta$ and every map $\Spec
    S\rightarrow X$, the kernel of $\Mod_S^\alpha\rightarrow\Mod_S^\beta$ is generated
    by compact objects of $\Mod_S^\alpha$.

\end{definition}

In the literature, what is called here the telescope hypothesis is often called the
telescope conjecture. Since it is false in general, hypothesis seems more appropriate.

\begin{lemma}\label{lem:unital}
    If $\Mscr$ is a symmetric monoidal stable presentable $\infty$-category that is compactly generated
    by its unit, and if $\Cscr$ is a compactly generated $\Mscr$-linear category, then $\Cscr$ satisfies the $\Mscr$-linear telescope hypothesis if and only
    if $\Cscr$ satisfies the telescope hypothesis.
    \begin{proof}
        If $\Cscr$ satisfies the telescope hypothesis, then it satisfies the less
        restrictive $\Mscr$-linear telescope hypothesis. Conversely, we claim that any
        localization $j:\Cscr\rightleftarrows\Dscr:j_\rho$ is automatically $\Mscr$-linear.
        It suffices to show that $\ker(j)$ is closed under tensor product with $\Mscr$.
        Contemplation of the following four facts completes the proof.
        The localizing subcategory $\ker(j)\subseteq\Cscr$ is closed under homotopy colimits
        by definition. The tensor
        product preserves homotopy colimits in each variable. The symmetric monoidal stable
        $\infty$-category $\Mscr$ is generated under
        homotopy colimits by its unit $\mathds{1}_\Mscr$. Obviously, $\mathds{1}_\Mscr\otimes
        x\in\ker(j)$ for $x\in\ker(j)$.
    \end{proof}
\end{lemma}

The conclusion is closely related to an
observation of Thomason~\cite{thomason-triangulated}*{Corollary 3.11.1(a)}: every thick subcategory of $\Perf(\Spec R)$ for a
commutative ring $R$ is automatically a $\otimes$-ideal.

Recall that the $\Mod_X$-linear category of global sections of $\StMod^\alpha$ is
\begin{equation*}
    \Mod^\alpha_X=\lim_{\Spec S\rightarrow X}\Mod_S^\alpha.
\end{equation*}
The following theorem allows passage back and forth between the linear and the
stacky telescope hypotheses.

\begin{theorem}\label{thm:stacks}
    If $X$ is a quasi-compact and quasi-separated derived scheme over $R$ and
    $\alpha:X\rightarrow\StCat^\desc$, then
    $\StMod^\alpha$ satisfies the stacky telescope hypothesis if and only if $\Mod^\alpha_X$
    satisfies the $\Mod_X$-linear telescope hypothesis.
    \begin{proof}
        Suppose first that $\StMod^\alpha$ satisfies the stacky telescope hypothesis, and let
        $j:\Mod^\alpha_X\rightarrow\Dscr$ be a $\Mod_X$-linear smashing localization of
        $\Mod_X$-linear categories.
        Because $\StMod^\alpha\rightarrow 0$ is a smashing localization, the stacky
        telescope hypothesis for $\StMod^\alpha$ says that
        $\Mod^\alpha_S$ is compactly generated for every $\Spec
        S\rightarrow X$.
        By Theorem~\ref{thm:lgcptg}, it follows that
        $\Mod^\alpha_X$ is compactly generated, which in turn implies that $\Dscr$ is
        compactly generated. Indeed, because $j_\rho$ preserves coproducts, adjunction
        implies that $j$ preserves compact objects. Since $j_\rho$ is fully faithful, we can
        take as a set of compact generators of $\Dscr$ the image under $j$ of a set of
        compact generators of $\Mod_X^\alpha$. Define
        \begin{equation*}
            \Mod^\beta_S=\Mod_S\otimes_{\Mod_X}\Dscr.
        \end{equation*}
        Since $\Dscr$ is compactly generated, it follows from~\cite{dag11}*{Corollary 6.11}
        that $\StMod^\beta$ is indeed a stack of linear categories. We claim that
        $\StMod^\alpha\rightarrow\StMod^\beta$ is a smashing localization of stacks. But,
        this is clear because the adjoint
        \begin{equation*}
            j_{S,\rho}:\Mod^\beta_S\rightarrow\Mod^\alpha_S
        \end{equation*}
        can be written as
        \begin{equation*}
            \id_{\Mod_S}\otimes_{\Mod_X}j_\rho:\Mod_S^\beta\we\Mod_S\otimes_{\Mod_X}\Dscr\rightarrow\Mod_S\otimes_{\Mod_X}\Mod_X^\alpha\we\Mod_S^\alpha.
        \end{equation*}
        Since $j_\rho:\Dscr\rightarrow\Mod_X^\alpha$ preserves coproducts, so does
        $j_{S,\rho}$. Similarly, $\id_{\Mod_S}\otimes_{\Mod_X}j_\rho$ is fully faithful
        because $j_\rho$ is. By the stacky telescope hypothesis for $\StMod^\alpha$, it follows
        that each $\ker(j_S)$ is compactly generated by objects in $\Mod_S^\alpha$. Again,
        using Theorem~\ref{thm:lgcptg},
        it follows that the $\infty$-category of sections
        of the kernel stack is compactly generated. By Corollary~\ref{cor:perfect}, the
        compact objects of $\ker(j)$ are perfect. Thus, they are \'etale locally compact in
        $\Mod_X^\alpha$. So, they are perfect, and hence compact, in $\Mod_X^\alpha$. Thus,
        $\ker(j)$ is compactly generated by compact objects of $\Mod_X^\alpha$.

        Now, assume that $\Mod^\alpha_X$ satisfies the $\Mod_X$-linear telescope hypothesis. Let
        $\StMod^\alpha\rightarrow\StMod^\beta$ be a smashing localization of stacks.
        Consider the induced functor $\Mod^\alpha_X\rightarrow\Mod^\beta_X$, which we claim
        is a $\Mod_X$-linear smashing localization. Since
        mapping spaces can be computed locally, the right adjoint $j_{X,\rho}$ is
        fully faithful. Similarly, if $\colim_I y_i\rwe y$ is a colimit diagram in
        $\Mod_X^\beta$, then the natural map $\colim_I j_{X,\rho}(y_i)\rightarrow
        j_{X,\rho}(y)$ is locally an equivalence. Thus, it is an equivalence in
        $\Mod_X^\alpha$, so that $j_{X,\rho}$ preserves coproducts. This proves the claim.
        Now, let $\Kscr_X=\ker(j_X)$, and set $\Kscr_S=\Mod_S\otimes_{\Mod_X}\Kscr_X$ for every $\Spec
        S\rightarrow X$. Since the $\Mod_X$-linear telescopy hypothesis applied to $j$ says
        that $\Kscr_X$ is generated by compact objects of $\Mod_X^\alpha$, it follows that
        $\Kscr_X$ is,
        in particular, compactly generated, so that $\Kscr$ defines a stack of linear
        categories by~\cite{dag11}*{Corollary 6.11}.
        It follows immediately that $\Kscr_S$ is generated by compact objects of
        $\Mod_S^\alpha\we\Mod_S\otimes_{\Mod_X}\Mod_X^\alpha$. But, $\Kscr_S$ is also the
        kernel of $j_S$, as tensoring with $\Mod_S$ over $\Mod_X$ preserves the exact
        sequence
        \begin{equation*}
            \Kscr_X\rightarrow\Mod_X^\alpha\rightarrow\Mod_X^\beta.
        \end{equation*}
        Therefore, the kernel of $j_S$ is generated by compact objects of $\Mod_S^\alpha$
        for every $\Spec S\rightarrow X$. So, $\StMod^\alpha$ satisfies the stacky telescope
        hypothesis.
    \end{proof}
\end{theorem}

\section{Telescopy for triangulated categories}\label{sec:triangulated}

Just as for stable $\infty$-categories, several notions of telescopy for
triangulated categories exist. The first is straightforward. A triangulated category $\Tscr$
satisfies the triangulated telescopy hypothesis ($\Tscr$TH) if every smashing localization
\begin{equation*}
    j:\Tscr\rightleftarrows\Tscr':j_{\rho}
\end{equation*}
has a kernel
generated by compact objects of $\Tscr$. Note that by definition every smashing localization
is a Bousfield localization, so it is determined by the kernel of $j$.

Now, suppose as in the introduction that $\Tscr$ and $\Uscr$ are compactly generated
triangulated categories with all coproducts, that $\Tscr$ is a $\otimes$-triangulated
category, and that there is a $\otimes$-product $\otimes:\Tscr\times\Uscr\rightarrow\Uscr$
that preserves coproducts in each variable separately.
Then $\Uscr$ satisfies the $\otimes$-telescope hypothesis under the action of $\Tscr$
($\otimes$TH) if
every smashing localization where the localizing subcategory is closed under the action of
$\Tscr$ is
generated by compact objects of $\Uscr$.

If $\Mscr$ is a compactly generated symmetric monoidal stable presentable $\infty$-category
and $\Cscr$ is $\Mscr$-linear, then $\Ho(\Mscr)$ and $\Ho(\Cscr)$ satisfy the hypotheses on
$\Tscr$ and $\Uscr$ above.

\begin{lemma}\label{lem:hotimes}
    The stable presentable $\infty$-category $\Cscr$ satisfies the $\Mscr$-linear telescope hypothesis if and only if the
    $\otimes$-telescope hypothesis holds for $\Ho(\Cscr)$ under the action of
    $\Ho(\Mscr)$.
    \begin{proof}
        Suppose that $\Ho(\Cscr)$ satisfies the $\otimes$-telescope hypothesis, and let
        $j:\Cscr\rightleftarrows\Dscr:j_\rho$ be an $\Mscr$-linear smashing localization.
        Then,
        \begin{equation*}
            \Ho(j):\Ho(\Cscr)\rightleftarrows\Ho(\Dscr):\Ho(j_\rho)
        \end{equation*}
        is a smashing localization. Moreover, the kernel of $\Ho(j)$ is
        $\Ho(\Mscr)$-closed,
        since $j$ is $\Mscr$-linear. Therefore, by the $\otimes$-telescope hypothesis for
        $\Ho(j)$, this kernel is generated by compact objects of $\Ho(\Cscr)$. It follows
        that the kernel of $j$ is generated by compact objects of $\Cscr$.

        Now, suppose that $\Cscr$ satisfies the $\Mscr$-linear telescope hypothesis, and let
        \begin{equation*}
            h:\Ho(\Cscr)\rightleftarrows\Tscr:h_\rho
        \end{equation*}
        be a smashing localization where $\ker(h)$ is $\Ho(\Mscr)$-closed. Let $\Kscr$ be the
        full subcategory of $\Cscr$ consisting of objects $x$ whose homotopy class in
        $\Ho(\Cscr)$ is contained in $\ker(h)$. By hypothesis, $\Kscr$ is closed under
        tensoring with objects of $\Mscr$. From the existence of the adjoint $h_\rho$, it
        follows that $\ker(h)$ is well-generated in the sense of triangulated categories
        (see for instance~\cite{krause-localization}). Hence, $\Kscr$ is a stable
        presentable $\infty$-category by~\cite{ha}*{Lemma 1.4.4.2}. It follows that $\Kscr$
        is also $\Mscr$-linear
        category (using this closure property and the fact that $\Mscr$ is itself). Since
        $\Kscr$ is presentable and $\Kscr\rightarrow\Cscr$ preserves coproducts, it follows
        that the inclusion has a right adjoint by the adjoint functor theorem for
        presentable $\infty$-categories~\cite{htt}*{Corollary 5.5.2.9}. Thus, there is a localization $\Cscr\rightarrow\Dscr$
        with kernel $\Kscr$, which can be identified with the map from $\Cscr$ to the
        cofiber of $\Kscr\rightarrow\Cscr$ in the $\infty$-category of stable presentable
        $\infty$-categories. By construction, $\Ho(\Dscr)\we\Tscr$, and the localization is
        smashing, since this can be checked at the level of homotopy categories. Thus,
        $\Kscr$ is generated by compact objects of $\Cscr$, and so $\ker(j)$ is
        generated by compact objects of $\Ho(\Cscr)$, as desired.
    \end{proof}
\end{lemma}

\begin{lemma}\label{lem:hotel}
    A stable $\infty$-category $\Cscr$ satisfies the telescope
    hypothesis if and only if its homotopy category $\Ho(\Cscr)$ satisfies the triangulated telescope
    hypothesis.
    \begin{proof}
        This is left to the reader. It is straightforward using the techniques
        of~\cite{ha}*{Section 1.4.4} and similar to the proof of the previous lemma.
    \end{proof}
\end{lemma}

In the Figure~\ref{fig:implications}, the implications are compiled between the various
telescope hypotheses. This paper is essentially about those on the first row. However, the
telescope conjecture originally arose in the setting of triangulated categories, so it is
useful to be able to go back and forth from that world to this one.
The most important conceptual arrow in the figure is the implication proved in the previous
section that $\textrm{LTH}$ is
equivalent to $\textrm{STH}$.

\begin{figure}[h]
    \centering
    \begin{equation*}
        \xymatrix{
            \textrm{TH} \ar@{<=>}[d] \ar@{=>}[r]  &   \textrm{LTH} \ar@{<=>}[d] \ar@{<=>}[r]   &   \textrm{STH}\\
            \Tscr\textrm{TH} \ar@{=>}[r]    &   \otimes\textrm{TH}  &
        }
    \end{equation*}
    \caption{The implications between various telescope hypotheses.}
    \label{fig:implications}
\end{figure}
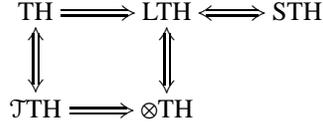

\begin{remark}
    In general it is difficult to lift constructions at the level of triangulated categories
    to the level of some model, be it a stable model category, a stable $\infty$-category,
    or a dg category. However, smashing localizations \emph{only} make sense in the presence
    of a Bousfield localization, and these are well enough behaved to be modeled. This is
    one reason why the classification of smashing localizations is easier than the
    classification of all localizing subcategories of a triangulated category.
\end{remark}

\section{The local-global principle}\label{sec:localglobal}

The next theorem is the main result of the paper. In the proof, note that if
$i:\Kscr\rightarrow\Mod_S^\alpha$ is a fully faithful inclusion of $S$-linear categories
with a right adjoint $i_\rho$ and if $\Kscr$ is generated by a set of objects that are
compact in $\Mod_S^\alpha$, then \emph{every} compact object of $\Kscr$ is compact when
viewed as an object of $\Mod_S^\alpha$.

\begin{theorem}\label{thm:localglobal}
    Let $X$ be a quasi-compact and quasi-separated derived scheme, and suppose that
    $\StMod^\alpha$ is a stack of linear categories on $X$. If
    there is an \'etale cover $f:U\rightarrow X$ such that $\Mod^\alpha_U$ satisfies the
    $\Mod_U$-linear
    telescope hypothesis, then $\Mod^\alpha_X$ satisfies the $\Mod_X$-linear telescope hypothesis.
    \begin{proof}
        By Theorem~\ref{thm:stacks}, it is enough to show that $\StMod^\alpha$ satisfies the
        stacky telescope hypothesis.  Let $$j:\StMod^\alpha\rightleftarrows\StMod^\beta:j_\rho$$ be a smashing localization of
        stacks, and consider the stack of kernels $\Kscr$; that is, $\Kscr_T=\ker(j_T)$ for
        $\Spec T\rightarrow X$. We must show that $\Kscr_T$ is generated by
        compact objects in $\Mod^\alpha_T$ for every $\Spec T\rightarrow X$.
        Fix a map $g:\Spec T\rightarrow X$, and consider the induced
        \'etale cover $f_T:U\times_X\Spec T\rightarrow\Spec T$ given by pulling back
        $f:U\rightarrow X$. As Theorem~\ref{thm:stacks} says that $\StMod^{f^*\alpha}$
        satisfies the stacky telescope hypothesis (over $U$ in this case), it follows
        immediately that $\StMod^{g_U^*\alpha}$ satisfies the stacky telescope hypothesis
        (over $U\times_X\Spec T$).
        In particular, by quasi-compactness and quasi-separatedness, there is an affine
        hypercover $\Spec S^\bullet\rightarrow\Spec T$ such that each $\Mod^\alpha_{S^k}$
        satisfies the $S^k$-linear telescope hypothesis. In other words, each $\Kscr_{S^k}$
        is compactly generated by objects of $\Mod^\alpha_{S^k}$.
        Since $\Kscr_T$ is a $T$-linear category with descent, the vertical arrows of
        the commutative diagram
        \begin{equation*}
            \xymatrix{
                \Kscr_T \ar[r] \ar[d]  &   \Mod^\alpha_T \ar[d]\\
                \lim_\Delta\Kscr_{S^\bullet} \ar[r]  &
                \lim_\Delta\Mod^\alpha_{S^\bullet}
            }
        \end{equation*}
        are equivalences. In particular, $\Kscr_T$ is \'etale locally compactly generated,
        so that it is compactly generated by Theorem~\ref{thm:lgcptg}. It suffices now to show
        that the inclusion functor $i:\Kscr_T\rightarrow\Mod^\alpha_T$ preserves compact
        objects. Let $x_T$ be a compact object of $\Kscr_T$. Each restriction
        $x_{S^k}$ is compact in $\Kscr_{S^k}$ by Corollary~\ref{cor:perfect}, which means that
        $i(x_{S^k})$ is compact in $\Mod^\alpha_{S^k}$ by hypothesis.
        It follows that $i(x)$ is perfect and hence compact, as desired.
        Therefore, $\Kscr_T$ is generated by compact objects of $\Mod_T^\alpha$.
    \end{proof}
\end{theorem}

\begin{corollary}\label{cor:az}
    If $X$ is a noetherian scheme and $\alpha\in\Br'(X)$, then $\Mod_X^\alpha$ satisfies the
    $\Mod_X$-linear telescope hypothesis.
    \begin{proof}
        In this case, one can take an \'etale cover $\coprod_i\Spec S_i\rightarrow X$ such
        that the restriction of $\alpha$ to each $\Spec S_i$ is trivial and such that $S_i$ is
        noetherian. The result of Hopkins and Neeman~\cite{neeman-chromatic} says that the
        telescope conjecture holds for $\Drm(S_i)$ and hence $\Mod_{S_i}$ by
        Lemma~\ref{lem:hotel}. In particular, $\Mod_{S_i}$ satisfies the $S_i$-linear
        telescope hypothesis. Thus, by the theorem, the $\Mod_X$-linear telescope hypothesis
        holds for $\Mod_X^\alpha$.
    \end{proof}
\end{corollary}

It follows from the corollary that $\Drm_{\qc}(X,\alpha)$ satisfies the telescope hypothesis
for localizations whose kernel is closed under tensor product with complexes in
$\Drm_{\qc}(X)$.

\begin{corollary}
    If $X$ is a noetherian scheme, then $\Drm_{\qc}(X)$
    satisfies the $\otimes$-telescope hypothesis.
    \begin{proof}
        This follows from the previous corollary, with $\alpha=0$, and
        Lemma~\ref{lem:hotimes}.
    \end{proof}
\end{corollary}

The second corollary was obtained previously, by~\cite{atjlss}, \cite{balmer-favi},
\cite{hovey-palmieri-strickland}, and~\cite{stevenson-support}. In flavor, the method used
here is most similar to that
of Balmer and Favi, although, as the first corollary demonstrates for $\alpha\neq 0$,
Theorem~\ref{thm:localglobal} has much broader consequences. In fact, the first corollary holds
even for $\alpha$ in the larger derived Brauer group of $X$ (see~\cite{toen-derived}). The
proof is no different. The power of our method is that we can use \'etale locality to check
for telescopy, rather than just Zariski local methods\footnote{Using the results of
To\"en~\cite{toen-derived}, these results can be extended to give an fppf local-global
principle for telescopy. However, without any applications in mind, this story is omitted.}.

As a third corollary, we obtain a classification result for the smashing
$\otimes$-localizations of $\Drm_{\qc}(X,\alpha)$.

\begin{corollary}
    Let $X$ be a quasi-compact and quasi-separated scheme, and let $\alpha\in\Br'(X)$.
    There is a bijection between the smashing $\otimes$-localizing subcategories of $\Drm_{\qc}(X,\alpha)$ under
    $\Drm_{\qc}(X)$ and the specialization closed subsets of $X$ that can be written as
    unions of closed subschemes of $X$ with quasi-compact complements.
    \begin{proof}
        To any smashing $\otimes$-localizing subcategory $\Drm$ of $\Drm_{\qc}(X,\alpha)$, we can
        associate the specialization closed subset of $X$ consisting of the union of all
        supports of all perfect complexes in $\Drm$. Since the support of any
        $\alpha$-twisted perfect complex is a closed subset with quasi-compact
        complement (see for instance~\cite{thomason-triangulated}),
        we obtain one direction of the correspondence. To get the other direction, we
        use Thomason's result~\cite{thomason-triangulated} that this is true when $\alpha=0$.
        It is known, for instance by To\"en~\cite{toen-derived}, that $\Drm_{\qc}(X,\alpha)$ is generated by a
        single $\alpha$-twisted perfect complex, say $E$. Let $V\subseteq X$ be a specialization closed
        subset, written as $V=\bigcup_{i\in I}V_i$, where $V_i\subseteq X$ is closed with
        quasi-compact complement. Then, for each $i$ there is a perfect complex $K_i$ in
        $\Drm_{\qc}(X)$ with support exactly $V_i$, and any such perfect complex generates
        $\Drm_{\qc,V_i}(X)$, the smashing subcategory of complexes supported on $V_i$. The
        collection of objects $K_i\otimes E$ generates a smashing localizing
        subcategory of $\Drm_{\qc}(X,\alpha)$ whose support is precisely $V$. It thus suffices to
        show that any two smashing $\otimes$-localizing subcategories of
        $\Drm_{\qc}(X,\alpha)$ supported on $V$ are equivalent. We can reduce to the case
        that $V=V_1$ is irreducible with quasi-compact complement. So, assume that $\Drm_1$
        and $\Drm_2$ are smashing localizing subcategories of $\Drm_{\qc}(X,\alpha)$ that
        are closed under tensoring with objects of $\Drm_{\qc}(X)$, and assume moreover that
        the supports of $\Drm_1$ and $\Drm_2$ are both identically $V$. The dual $E^\vee$ of
        $E$ is a $(-\alpha)$-twisted perfect complex. Note that the (derived) tensor product
        of an $\alpha$-twisted complex and a $\beta$-twisted complex is an
        $(\alpha+\beta)$-twisted complex. The
        $\otimes$-localizing subcategories generated by $\Drm_1\otimes E^\vee$ and
        $\Drm_2\otimes E^\vee$ in $\Drm_{\qc}(X)$ have support exactly $V$, and hence, by
        Thomason's result, coincide. It follows that the $\otimes$-closed localizing
        subcategories generated by $\Drm_1\otimes E^\vee\otimes E$ and $\Drm_1\otimes E^\vee\otimes
        E$ agree in $\Drm_{\qc}(X,\alpha)$. But, $E^\vee\otimes E$ is
        a perfect generator of $\Drm_{\qc}(X)$ (see~\cite{toen-derived}*{Definition 2.1}),
        so $\Drm_1\otimes E^\vee\otimes E$ generates $\Drm_1$,
        and similarly for $\Drm_2$. Hence, $\Drm_1=\Drm_2$.
    \end{proof}
\end{corollary}

Now, we consider some examples.

\begin{example}\label{ex:dg}
    Consider a singular noetherian affine scheme $X=\Spec S$ with
    a non-zero class $\alpha\in\Hoh^1_{\et}(X,\ZZ)$ (in which case $X$ is not normal). For instance, one can take
    $S=k[x,y,z]/(y^2-x^3+x^2)$. By~\cite{toen-derived}, $\Mod_X^\alpha\we\Mod_A$
    for some derived Azumaya $S$-algebra $A$. By construction, $A$
    cannot be derived equivalent to an ordinary associative algebra, for otherwise
    $\alpha\in\Br(X)$. Nevertheless, the
    $S$-linear telescopy hypothesis holds for $\Mod_A$ by the theorem. It follows
    that the telescope hypothesis holds for $\Mod_A$ and hence that the
    triangulated telescope hypothesis holds for $\Drm(A)$. To our knowledge, this is the first
    example of any version of the telescope hypothesis for a truly derived dg algebra.
\end{example}

\begin{example}
    In~\cite{dwyer-palmieri}, Dwyer and Palmieri give an example of a non-noetherian scheme for
    which the telescope hypothesis holds, namely the truncated polynomial ring on infinitely
    many generators $$\Spec k[t_1,t_2,\ldots]/(t_i^{n_i})$$ where $n_i\geq 2$ for all $i$. The theorem says
    that for an Azumaya algebra over this ring, the telescope hypothesis holds. Any
    such Azumaya algebra is induced from a central simple algebra over $k$. But, this fact seems not
    to lead to an immediate proof of telescopy.
\end{example}

\section{Classifying stacks and gerbes}\label{sec:gerbes}

In this section, a proof is given of telescopy
for the derived category of gerbes and of classifying spaces of finite group schemes.
These cover two of the most important cases of Deligne-Mumford stacks. For
instance, the components of the moduli stack of semistable vector bundles on a smooth
projective surface are
abelian gerbes over noetherian schemes, so the results below apply.

\begin{theorem}\label{thm:classifying}
    Let $X$ be a noetherian scheme, and
    let $\Gscr\rightarrow X$ be a finite \'etale group scheme such that the fiber over every
    point $x\in X$ is of order prime to
    the characteristic of $k(x)$. Then, $\Mod_{\Brm\Gscr}$ satisfies the $\Mod_X$-linear telescopy hypothesis (and
    hence the $\Mod_{\Brm\Gscr}$-linear telescope hypothesis),
    where $\Brm\Gscr$ is the classifying stack of $\Gscr$ over $X$. 
    \begin{proof}
        Let
        \begin{equation*}
            \coprod_i\Spec S_i\rightarrow X
        \end{equation*}
        be an \'etale cover such that $\Gscr_{S_i}=\Gscr\times_X\Spec S_i$ is a constant
        finite group scheme. Then, the restriction of $\Mod_{\Brm\Gscr}$ to $\Spec S_i$ is
        $$\Mod_{\Brm\Gscr_S}\we\Mod_{S[G]}.$$ If $\overline{x}$ is a geometric point of
        $\Spec S$, then $k(\overline{x})[G]$ is a product of matrix algebras (since the order of $G$ is prime
        to the characteristic of $k(\overline{x})$). This product does not depend on
        the geometric point on the connected components of $X$. Therefore, using for example the arguments of~\cite{ag}*{Section 5.3}, it
        follows that $S[G]$ is \'etale locally a product of matrix algebras over central
        separable extensions of $S$.
        Since $S$ is noetherian, $\Mod_{S[G]}$ \'etale locally satisfies the
        linear telescope hypothesis by Corollary~\ref{cor:az}. But, this implies that $\Mod_{S[G]}$ satisfies the
        $S$-linear telescope hypothesis by Theorem~\ref{thm:localglobal}, and hence that
        $\Mod_{\Brm\Gscr}$ satisfies the $\Mod_X$-linear telescope hypothesis by the same
        theorem.
    \end{proof}
\end{theorem}

Recall that a tame Deligne-Mumford stack is one whose stabilizer groups have order prime to
the residue characteristics. The classifying stacks appearing in theorem are examples.

\begin{corollary}
    Let $\Xscr$ be a separated noetherian tame Deligne-Mumford stack
    whose stabilizers groups are locally constant, and assume that the coarse moduli space
    $X$ of $\Xscr$ is a noetherian scheme.
    Then $\Mod_\Xscr$ satisfies the $\Mod_X$-linear telescope hypothesis (and
    hence the $\Mod_\Xscr$-linear telescope hypothesis).
    \begin{proof}
        In this case, $\Xscr\rightarrow X$ is \'etale locally of the form $[\Spec
        T/G]\rightarrow\Spec S$, where $G$ is a finite group acting on $\Spec T$ with
        constant stabilizer $H$, by~\cite{abramovich-olsson-vistoli}*{Theorem 3.2}. It follows that $[\Spec T/G]$ is equivalent to the
        classifying stack of $H$ over $\Spec T^G$. But, $\Spec T^G$ is also noetherian, by
        hypothesis, so that the corollary follows from the previous theorem.
    \end{proof}
\end{corollary}

If more was known about the derived categories of algebraic spaces, then the assumption on
the coarse moduli space could possibly be dropped in the corollary. In particular, we are led to ask
the following question.

\begin{question}
    Does the $\Mod_X$-linear telescope hypothesis hold for the derived category of a noetherian algebraic
    space $X$?
\end{question}

In the non-tame case, it is still possible to say something, at least when the stabilizers
are abelian. Indeed, in that case, the group algebras $R[G]$ are in fact commutative and
noetherian, whence telescopy follows from Neeman's result. This is summarized in the next
proposition, which extends Theorem~\ref{thm:classifying}. There are analogs of the
corollaries as well, although we will leave their formulation to the reader.

\begin{proposition}\label{prop:abelian}
    Let $\Ascr$ be a finite \'etale abelian group scheme over a noetherian scheme $X$. Then, 
    $\Mod_{\Brm\Ascr}$ satisfies the $\Mod_X$-linear telescope hypothesis (and hence the
    $\Mod_{\Brm\Ascr}$-linear telescope hypothesis.
    \begin{proof}
        Indeed, \'etale locally on $X$, $\Ascr$ is a constant abelian group. If $\Spec
        S\rightarrow X$ is a map where $\Ascr_S$ is the constant abelian group scheme $A$, then
        $\Mod_{\Brm\Ascr_S}\we\Mod_{S[A]}$. But, $S[A]$ is a commutative noetherian ring, so that
        the $S$-linear telescope hypothesis holds for $\Mod_{S[A]}$. The rest of the
        proof follows now familiar lines.
    \end{proof}
\end{proposition}

\begin{corollary}
    If $X$ is a noetherian scheme and $\Xscr\rightarrow X$ is a finite abelian gerbe, then $\Mod_\Xscr$ satisfies the
    $\Mod_X$-linear telescope hypothesis (and hence the $\Mod_\Xscr$-linear telescope
    hypothesis).
    \begin{proof}
        In this case, $\Xscr\rightarrow X$ is \'etale locally on $X$ the classifying stack
        of a finite \'etale abelian group scheme. The corollary follows from the application of
        Proposition~\ref{prop:abelian} followed by Theorem~\ref{thm:localglobal}.
    \end{proof}
\end{corollary}

\begin{example}
    Suppose that $X$ is a smooth projective surface over a field, and that $\Mscr$
    is the moduli stack of geometrically stable vector bundles on $X$ of rank $r$, determinant $L$, and second
    Chern class $c\in\ZZ$. Then, the $\Mod_M$-linear telescope hypothesis holds for
    $\Mod_\Mscr$, where $M$ is the coarse moduli space of $\Mscr$. In fact, when $\Xscr\rightarrow X$ is a $\mu_n$-gerbe, this is true for
    $\Mscr_\Xscr$ as well, where $\Mscr_\Xscr$ is the moduli stack of geometrically stable $\Xscr$-twisted vector
    bundles of rank $r$, determinant $L$, and second Chern class $c$. For details on these
    stacks, see~\cite{lieblich-moduli}.
\end{example}

\begin{example}
    For a final example, let $\Xscr$ be a smooth Deligne-Mumford stack over $\CC$ of
    dimension at most $3$, with coarse moduli space a noetherian scheme $X$. Assume also
    that the canonical bundle of $\Xscr$ is trivial.
    This is precisely the situation in
    which the Bridgeland-King-Reid theorem~\cite{bridgeland-king-reid} holds. Thus, 
    the coarse moduli space $X$ has a crepant resolution, say $V\rightarrow X$, and there is a derived
    equivalence $\Mod_V\we\Mod_{\Xscr}$. The equivalence turns $\Mod_{\Xscr}$ into a
    $\Mod_V$-linear category, and since $V$ is a noetherian scheme, it follows from the
    previous section that $\Mod_{\Xscr}$ satisfies the $\Mod_V$-linear telescope hypothesis.
\end{example}

Given the numerous positive results in this section, the next question is rather natural.

\begin{question}
    Does the $\Mod_{\Xscr}$-linear telescope hypothesis hold for $\Mod_{\Xscr}$ when $\Xscr$
    is a noetherian Deligne-Mumford stack?
\end{question}

Another positive answer is provided by Dell'Ambrogio and
Stevenson~\cite{dellambrogio-stevenson}, who establish the linear telescope hypothesis for
the derived categories of weighted projective stacks.

The question is especially important when $\Xscr$ has a coarse moduli scheme $X$. If
moreover $\Xscr$ is smooth, a
recent paper of Dubey and Mallick~\cite{dubey-mallick} together with a positive answer to
the question would produce a classification of all $\otimes$-closed smashing localizations of
$\Drm_{\qc}(\Xscr)$: they would be in bijection with specialization closed subsets of $X$.
In particular, the theorems and statements of this section all lead to classification
theorems. We end with one example of such a classification theorem.

\begin{theorem}
    Let $X$ be a smooth scheme of finite type over a field $k$, and let $\Gscr\rightarrow X$
    be as in Theorem~\ref{thm:classifying}. Suppose that $\Xscr\rightarrow X$ is a
    $\Gscr$-gerbe (a stack over $X$ \'etale locally equivalent to $\Brm\Gscr$). Then, there is a
    bijection between the set of $\otimes$-closed smashing localizations of
    $\Drm_{\qc}(\Xscr)$ and the specialization closed subsets of $X$.
    \begin{proof}
        The coarse moduli space of $\Xscr$ is $X$, so
        by~\cite{dubey-mallick} there is an isomorphism
        $\mathrm{Spc}\,\Drm_{\perf}(X)\iso\mathrm{Spc}\,\Drm_{\perf}(\Xscr)$, where
        $\mathrm{Spc}$ denotes the spectrum of Balmer~\cite{balmer}. This means
        that there is a bijection between the thick $\otimes$-ideals in these two
        $\otimes$-triangulated categories. The result follows since, by
        Theorem~\ref{thm:classifying}, any $\otimes$-smashing localization is generated by
        its intersection with $\Drm_{\perf}(\Xscr)$ and from Thomason's classification
        of the thick $\otimes$-ideals of $\Drm_{\perf}(X)$~\cite{thomason-triangulated}.
    \end{proof}
\end{theorem}


%
%
%
%
%

\end{document}